\documentclass[11pt,reqno]{amsart}

\usepackage{amsmath, amsthm, amssymb}
\usepackage{amsfonts}
\usepackage[ansinew]{inputenc}
\usepackage[dvips]{epsfig}
\usepackage{graphicx}
\usepackage[english]{babel}
\usepackage{tikz}
\usepackage{enumerate}
\usepackage{hyperref}

\theoremstyle{plain}
\newtheorem{thm}{Theorem}[section]
\newtheorem{cor}[thm]{Corollary}
\newtheorem{lem}[thm]{Lemma}

\newtheorem{example}[thm]{Example}

\theoremstyle{definition}
\newtheorem{defi}[thm]{Definition}

\theoremstyle{remark}
\newtheorem{rem}[thm]{Remark}

\numberwithin{equation}{section}

\newcommand{\R}{\mathbb{R}}

\newcommand{\tr}{\mathrm{Tr}}
\newcommand{\defeq}{\mathrel{\mathop:}=}

\newcommand{\average}{{\mathchoice {\kern1ex\vcenter{\hrule height.4pt
width 6pt depth0pt} \kern-9.7pt} {\kern1ex\vcenter{\hrule
height.4pt width 4.3pt depth0pt} \kern-7pt} {} {} }}

\def\R{\mathbb{R}}

\begin{document}

\title[The obstacle problem for a class of degenerate fully nonlinear operators]{The obstacle problem for a class of degenerate fully nonlinear operators}

\author{Jo\~ao Vitor Da Silva}

\address{Departamento de Matem\'atica - Instituto de Ci\^{e}ncias Exatas - Universidade de Bras\'{i}lia. Campus Universit\'{a}rio Darcy Ribeiro, 70910-900, Bras\'{i}lia - Distrito Federal - Brazil.}
\email{J.V.Silva@mat.unb.br}
\address{Instituto de Investigaciones Matem\'{a}ticas Luis A. Santal\'{o} (IMAS) - CONICET (Argentine), Ciudad Universitaria, Pabell\'{o}n I (1428) Av. Cantilo s/n - Buenos Aires}
\email{jdasilva@dm.uba.ar}

\author{Hern\'an Vivas}

\address{Instituto de Investigaciones Matem\'{a}ticas Luis A. Santal\'{o} (IMAS) - CONICET (Argentine), Ciudad Universitaria, Pabell\'{o}n I (1428) Av. Cantilo s/n - Buenos Aires}
\address{Centro Marplatense de Investigaciones matem\'aticas/Conicet, Dean Funes 3350, 7600 Mar del Plata, Argentina}
\email{havivas@mdp.edu.ar}

\keywords{Free boundary problems, degenerate elliptic equations.}

\subjclass[2010]{35R35, 35J70}

\begin{abstract}
We study the obstacle problem for fully nonlinear elliptic operators with an anisotropic degeneracy on the gradient:
\[
  \left\{ \begin{array}{rcll}
  \min\left\{f-|Du|^\gamma F(D^2u),u-\phi\right\}  & = & 0 & \textrm{ in } \Omega \\
  u & = & g & \textrm{ on } \partial \Omega
  \end{array}\right.
\]
for some degeneracy parameter $\gamma\geq 0$, uniformly elliptic operator $F$, bounded source term $f$, and suitably smooth obstacle $\phi$ and boundary datum $g$. We obtain existence/uniqueness of solutions and prove sharp regularity estimates at the free boundary points, namely $\partial\{u>\phi\} \cap \Omega$. In particular, for the homogeneous case ($f\equiv0$) we get that solutions are $C^{1,1}$ at free boundary points, in the sense that they detach from the obstacle in a quadratic fashion, thus beating the optimal regularity allowed for such degenerate operators. We also prove several non-degeneracy properties of solutions and partial results regarding the free boundary.

These are the first results for obstacle problems driven by degenerate type operators in non-divergence form and they are a novelty even for the simpler prototype given by an operator of the form $\mathcal{G}[u] = |Du|^\gamma\Delta u$, with $\gamma >0$ and $f \equiv 1$.
\end{abstract}

\maketitle

\section{Introduction} \label{sec.intro}

\subsection{Motivation and main proposals}

A classical problem from Mathematical Physics refers to the equilibrium position of an elastic membrane (whose boundary is held fixed) lying on top of a given body (an obstacle) under the action of contact and/or action-at-a-distance forces, \textit{e.g.} friction, tension, air resistance and gravity. Currently, this archetype model is often called an \textit{obstacle problem}. This kind of problem also arises in the context of optimal stopping problems in control theory, when one aims at minimizing/maximizing the cost/profit of a running functional, see \cite{Eva12}. The latter (probabilistic) approach corresponds to the ``non-variational'' mathematical formulation of the problem and leads, from this perspective, to the following model: find a profile $v$ which fulfills (in a suitable sense) the \textit{constrained free boundary problem}:
\[
  \left\{ \begin{array}{rcll}
  \min\left\{f-\mathfrak{L}\, v,\,u-\phi\right\}  & = & 0 & \textrm{ in } \,\,\,\Omega \\
  u & = & g & \textrm{ on } \,\,\, \partial \Omega.
  \end{array}\right.
\]
Here $v$ is the expected value of the cost/profit function (or the shape of the membrane in the physical problem), $\mathfrak{L}$ is an elliptic operator, representing, to some extent, the heterogeneity of the media, $\phi$ is an obstacle (or physical constraint), $f$ is a bounded source datum (\textit{e.g} a gravitational force acting on the membrane) and $g$ is a fixed boundary condition (see, \cite{Fig18}, \cite[Sec.3]{Ros18} and \cite[Ch.1]{Rod} and \cite[Ch.3]{Tei07} for more motivations and applications). From a variational view-point, solutions for such a free boundary problem can be obtained by minimizing the corresponding energy functional associated to $\mathfrak{L}$ (if this one enjoys a certain structure, see \textit{e.g.} \cite{ALS15}, \cite{C1998}, \cite{FKR-O} and \cite{Rod}) within an appropriate space of functions whereas non-variational techniques include Perron's method, penalization techniques, Comparison Principle tools, Weak Harnack and Harnack inequalities, Local Maximum Principle, see \cite{BlankTeka}, \cite{BLOP}, \cite{daSLR}, \cite{L}, \cite{LS01}, \cite{LP} and \cite{Teka} for some examples.

We should remember that obstacle type problems have attracted an increasing enthusiasm of the multidisciplinary scientific community for the last five decades or so. One of the main reasons, besides their intrinsic mathematical appeal that combines tools from regularity theory for PDEs, Geometric Measure Theory, Nonlinear Potential Theory and Harmonic Analysis, is that they are ubiquitous in Sciences, Mechanics, Engineering and Industry. In fact, problems as varied as flow through porous dam, cellular membranes' permeability, optimal stopping problems in Mathematical Finance, superconductivity of bodies in mean-field models in Physics are just some examples of phenomena that appear to be well described by these type of problems. We refer the reader to \cite{Fri}, \cite{KS}, \cite{PSU} or \cite{Rod} and the references therein for instrumental surveys concerning these and other examples.

Finally, progresses of such investigations concerning obstacle problems (with divergence structure) brought out not only significant theoretical breakthroughs, they also have proved to be of crucial importance in a wide range of models in affine areas like mathematical programming, control and optimization theory, game theoretical methods in PDEs among others, see \cite{Fig18}, \cite{Rod} and \cite{Ros18} for some examples and applications.

In the sequel, we will take a mathematical tour by the recent developments regarding inhomogeneous elliptic obstacle problems in \textit{non-divergence form}. In this direction, we must quote \cite{BLOP}, where the authors study the following elliptic obstacle problems:
$$
  \left\{
  \begin{array}{rcll}
  a^{ij}(x)D_{ij} u &\leq & f & \textrm{ in } \Omega \\
  (a^{ij}(x)D_{ij} u-f)(u-\psi) & = & 0 & \textrm{ in } \Omega \\
  u&\geq &\psi&\textrm{ in } \Omega\\
  u & = & 0 & \textrm{ on } \partial \Omega,
\end{array}\right.
$$
and
$$
  \left\{
  \begin{array}{rcll}
  F(x, D^2 u) &\leq & f & \textrm{ in } \Omega \\
  (F(x, D^2 u)-f)(u-\psi) & = & 0 & \textrm{ in } \Omega \\
  u&\geq &\psi&\textrm{ in } \Omega\\
  u & = & 0 & \textrm{ on } \partial \Omega,
\end{array}\right.
$$
where the coefficient matrix $(a^{ij}(x))_{i, j=1}^n$ and the fully nonlinear operator $F: \Omega \times \text{Sym}(n) \to \R$ are supposed to be uniformly elliptic (with $X \mapsto F(x, X)$ a convex mapping). In such a scenario, they investigate existence/uniqueness and regularity properties of solutions. Specifically, weighted Calder\'{o}n-Zygmund estimates for solutions of elliptic problems with discontinuous coefficients and irregular obstacles (see \textit{e.g.} \cite[Theorems 3.3 and 4.4]{BLOP} for details). They also establish Morrey type estimates for the Hessian, as well as H\"{o}lder continuity of the gradient of the solutions (see \cite[Theorems 5.1, 5.2 and 5.4]{BLOP} for more details). On the other hand, Blank and Teka in \cite{BlankTeka} (see also \cite{Teka} for a rather complete exposition) deal with strong solutions $w\geq 0$ of an obstacle problem of the form
$$
    \mathfrak{L}\, w(x) = a^{ij}(x)D_{ij} w(x) = \chi_{\{w>0\}} \quad \text{in} \quad B_1,
$$
In such a context, by assuming that $a^{ij} \in \text{VMO}(\Omega)$ (and uniform ellipticity), the authors prove existence of nontrivial solutions, non-degeneracy and optimal regularity of solutions. Finalizing this journey, the works \cite{L}, \cite{LS01} and \cite{LP} address a complete study on obstacle type problems in the fully nonlinear scenario with homogeneous obstacles and/or source terms and their corresponding regularity theories of solutions and free boundaries.

In spite of the fact that there is a huge amount of literature on obstacle problems in divergence form and their qualitative features (see \cite{ALS15}, \cite{C77}, \cite{C1998}, \cite{CK}, \cite{KS}, \cite{LS03},  \cite{Rod} and \cite{Tei07}), quantitative counterparts for non-variational (elliptic) models with degenerate nature are far less studied due to the rigidity of the structure of such operators (cf. \cite{daSLR} for an enlightening example). Therefore, the treatment of such free boundary problems requires the development of new ideas and modern techniques. This lack of investigation has been our main impetus in studying fully nonlinear obstacle type models with non-uniformly elliptic (anisotropic) structure, which focus on a modern, systematic and non-variational approach for such a general class of operators.

Taking into account the mathematical (physical) model used for the class of obstacle type problems is flexible and applicable to a broad spectrum of contexts, this manuscript is concerned with existence/uniqueness, regularity and geometric issues for problems governed by second order fully nonlinear elliptic equations of degenerate type as follows:
\begin{equation}\label{Eq1}
  \left\{ \begin{array}{rcll}
  \min\left\{f-|Du|^\gamma F(D^2u),u-\phi\right\}  & = & 0 & \textrm{ in } \Omega \\
  u & = & g & \textrm{ on } \partial \Omega,
\end{array}\right.
\end{equation}
where $\Omega \subset \R^n$ is a smooth, open and bounded domain, $\phi$ and $g$ are suitably smooth functions defined in $\Omega$ and $\partial \Omega$ respectively, $f$ is a continuous and bounded function in $\Omega$, $\gamma\geq 0$ and $F$ is a second order fully nonlinear uniformly elliptic operator. We recall that, for second order operators, \textit{uniform ellipticity} means that for any pair of matrices $X,Y\in\mathbb{R}^{n\times n}$
\begin{equation}\label{eq.pucci}
  \mathcal{M}_{\lambda, \Lambda}^-(X-Y)\leq F(X)-F(Y)\leq\mathcal{M}_{\lambda, \Lambda}^+(X-Y)
\end{equation}
where $\mathcal{M}_{\lambda, \Lambda}^-$ and $\mathcal{M}_{\lambda, \Lambda}^+$ are the \textit{Pucci extremal operators} given by
\[
   \mathcal{M}_{\lambda, \Lambda}^-(X)=\lambda\sum_{e_i>0}e_i+\Lambda\sum_{e_i<0}e_i\quad\textrm{ and }\quad \mathcal{M}_{\lambda, \Lambda}^+(X)=\Lambda\sum_{e_i>0}e_i+\lambda\sum_{e_i<0}e_i
\]
for some \emph{ellipticity constants} $0<\lambda\leq \Lambda< \infty$ (here $\{e_i\}_i$ are the eigenvalues of $X$). Throughout this work we will often refer to \eqref{Eq1} as the $(F, \gamma)$-obstacle problem.

It is noteworthy that our contributions extend (regarding non-variational scenario) \cite{ALS15}, as well as generalize, former seminal results (sharp regularity estimates) from \cite{BlankTeka}, \cite{BLOP} and \cite{Teka}, and to some extent, of those from \cite{L}, \cite{LS01} and \cite{LP} by making using of different approaches and techniques adapted to the general framework of the fully nonlinear (anisotropic) models.

Finally, to the best of the authors' knowledge, the results presented here comprise the first known results of obstacle problems driven by degenerate equations in non-divergence form, and they are new even for simpler (linear second order operators) such as the Laplacian, \textit{e.g.} for an operator of the form $\mathcal{G}[u] = |Du|^\gamma\Delta u$.

\subsection{Fully nonlinear operators and their main intricacies}

Degenerate operators with an anisotropic degeneracy on the gradient such as the one appearing in \eqref{Eq1} have interest both from the point of view of applied sciences and engineering and from a pure PDE perspective. We recommend the reader Birindelli and Demengel's works \cite{BerDem}, \cite{BeDe1}, \cite{BD2} and \cite{BD3} for a number of examples of degenerate fully nonlinear operators with similar structural properties. In fact, it is worth pointing out that the operator $\mathcal{G}(Du, D^2u) = |Du|^\gamma F(D^2u)$ is the simplest example of a more general class of operators dealt with in the aforementioned papers. Our results could be extended without (much) effort to that broader class (see, \cite{ART15} and \cite{daSLR}), but with decided to stick with $\mathcal{G}(Du, D^2u) = |Du|^\gamma F(D^2u)$ for the sake of simplicity and for ease of exposition.

It is worth highlighting that some of the major difficulties in dealing with our class of operators are: its non-divergence structure, in consequence, we are not allowed to make use of (nowadays) classical estimates from potential theory (cf. \cite{ALS15}, \cite{CK}), and the degeneracy character that forces diffusion properties (\textit{e.g.}, uniformly ellipticity of operator) to collapse along an \textit{a priori} unknown set of singular points of solution (a ``nonphysical free boundary''), namely the set:
$$
  \mathcal{S}_0 \defeq \{x\in \Omega : |Du(x)| = 0\}.
$$

These features produce significant constraints on the regularity that can be expected for solutions to such operators. Indeed, this is true even for non-degenerate and translation invariant operators. More precisely, it is known that viscosity solutions for fully nonlinear uniformly elliptic equations (with ``frozen'' coefficients)
\begin{equation}\label{EqF}
  F(x_0, D^2u)=0\quad\textrm{ in } \quad \Omega
\end{equation}
are locally $C^{1, \alpha_F}$, for a constant $\alpha_F \in (0,  1)$ that depends only on dimension and ellipticity constants (see \cite{CC95}). Of course, the primary question is whether solutions of such equation are smooth enough to be classical, i.e. at least $C^2$. Through the journey of finding these classical solutions, the result of Evans \cite{Ev82} and Krylov \cite{Kry83} was a pioneer paramount research on operators in non-divergence form; it states that under a concavity (or convexity) assumption on $F$, solutions of \eqref{EqF} are locally $C^{2, \alpha_0}$ for some $0 < \alpha_0 < 1$. The question of whether \emph{any} fully nonlinear elliptic operator would enjoy a $C^2$ \textit{a priori} regularity theory eluded the mathematical community for the last three decades and in effect, the Nadirashvili-Vl\u{a}du\c{t}'s counterexamples to $C^{1,1}$ regularity in \cite{NV1}, \cite{NV2}, \cite{NV3}, \cite{NV4} and \cite{NV5} give a negative answer to such a challenging question.

In our (degenerate) case the situation is considerably more involved, and thus present additional challenges.  We are dealing with equations of the form
\begin{equation}\label{eq.deg}
|D u|^\gamma F(D^2 u)=f(x)\quad\textrm{ in } \Omega.
\end{equation}
Simple examples (see, \cite[Section 3]{ART15} and \cite[Example 1]{IS}) show that, even for smooth right hand side, solutions are not better than $C_{\text{loc}}^{1,\frac{1}{\gamma+1}}$ in general (even if $F$ is concave/convex). For $f\in L^\infty(\Omega)$, Imbert and Silvestre showed in \cite[Theorem 1]{IS} that solutions to \eqref{eq.deg} are $C^{1, \alpha}$ for some (small) $\alpha$ in the interior of $\Omega$. Afterwards in \cite[Theorem 3.1]{ART15}, Ara\'ujo, Ricarte and Teixeira showed that in fact, given \emph{any} $\alpha\in (0,\alpha_F)\cap\left(0,\frac{1}{\gamma+1}\right]$ it is possible to show that $u\in C_{\text{loc}}^{1,\alpha}(\Omega)$, even for a more general class of operators satisfying natural structural conditions. As a matter of fact, from their result the (optimal) $C^{1,\frac{1}{\gamma+1}}$ interior regularity follows when $F$ is concave/convex (see, \cite[Corollary 3.2]{ART15}).

For our work, we consider the degeneracy exponent $\gamma>0$ to be given, as well as the boundary data $g\in C^{1,\beta}(\partial \Omega)$ and the obstacle $\phi\in C^{1,\beta}(\overline \Omega)$ for some $\beta\in(0,1]$. We define the optimal exponent
\begin{equation}\label{sharpEx}
\alpha \defeq \min\left\{\frac{1}{\gamma+1}, \beta\right\}.
\end{equation}

Further, we will assume that
\[
F\text{ is convex, satisfies \eqref{eq.pucci} and } F(0)=0.
\]
As discussed above, convexity is both a natural and important condition on $F$, otherwise the regularity constraints appear even in the uniformly elliptic, non degenerate and unconstrained case; as we are interested in the free boundary problem for the degenerate operator, we want to avoid unnecessary restrictions of this kind. Clearly all of our results are valid if $F$ is concave instead of convex. Also, the normalization $F(0)=0$ is assumed without loss of generality.

Then, we are interested in studying qualitative/quantitative features to the obstacle problem:
\begin{equation}\label{eq.obs}
\left\{ \begin{array}{rcll}
	u & \geq & \phi & \textrm{ in } \Omega \\
    |D u|^\gamma F(D^2u)   & \leq & f & \textrm{ in } \Omega \\
        |D u|^\gamma F(D^2u)  & = & f & \textrm{ in } \Omega\cap\{u>\phi\} \\
  u & = & g & \textrm{ on } \partial \Omega.
\end{array}\right.
\end{equation}

A more delicate mathematical matter on \eqref{eq.obs} (besides existence of solutions) is whether one can get $u\in C_{\textrm{loc}}^{1,\beta}(\Omega)$, that is if we can ``transmit'' the regularity across the ``physical'' free boundary $\partial \{u>\phi\}$. Crucial in the way to prove such a result is to obtain fine information about the behavior of the solution near free boundary points. Once this issue has been settled it is of interest to have some geometric information about free boundary itself; a natural first step is to prove that it has zero Lebesgue measure. We give a partial result in that direction as well.

Geometric regularity for equations as the ones studied here have been subject to much interest in the PDE/Analysis community in the last years, not only for its generality and several applications, but specially for its innate relation with a number of relevant free boundary problems in the literature (cf. \cite{ALS15}, \cite{CDV}, \cite{daSLR}, \cite{daSRS19}, \cite{daSS18} \cite{FKR-O} and \cite{LS03} for some variational and non-variational examples). For this reason, understanding the ``intrinsic geometry'' of the former model is an important step in comprising the behavior of solutions near their free boundary points.

Finally, it is worth noticing that obtaining quantitative/qualitative information of the solution close to the free boundary is a delicate matter, as well as it play a pivotal role for addressing a number of analytic and geometric issues in free boundary problems such as blow-up analysis, Liouville type results, free boundary regularity, geometric growth estimates, just to mention a few (see, \cite{BlankTeka}, \cite{daSLR}, \cite{daSRS19}, \cite{daSS18} and \cite{L} for some examples).

\subsection{Statement of main results}

In this section we present the main results in this manuscript. The sharp regularity exponent $\alpha$ is going to be fixed throughout and is defined by \eqref{sharpEx}. Our first main result assures that we are able to obtain a viscosity solution to $(F, \gamma)$-obstacle problem in $\Omega$ under suitable assumptions on the data and, moreover, these solutions enjoy a ``basic'' (in the sense that is not optimal) regularity estimate.

\begin{thm}[{\bf Existence of solutions with basic regularity}]\label{Thm1} Let $\phi \in C^{1, 1}(\Omega)$, $f \in C^0(\overline{\Omega})$ and $g\in C^{1,\beta}(\partial\Omega)$, $\beta\in(0,1]$. Then, there exists a viscosity solution $u$ to the $(F, \gamma)$-obstacle problem in $\Omega$. Moreover, $u \in C^{1,\beta^{\prime}}(\Omega)$ for some $\beta^{\prime}\ll 1$ and
\[
\displaystyle \|u\|_{C^{1, \beta^{\prime}}(\Omega)} \leq C\left(n, \lambda, \Lambda, \gamma, \|\phi\|_{C^{1, 1}(\Omega)}, \|f\|_{L^\infty(\Omega)}, \|g\|_{C^{1,\beta}(\partial\Omega)}\right).
\]
\end{thm}

\begin{rem}
Alternatively, one could prove existence of solutions via Perron's method (see, for instance \cite{CIL} or \cite{L}). This approach has the advantage of not requiring any smoothness on the obstacle other than continuity and the downside that it produces a solutions which is just continuous. Since from our perspective the main interest in this problem lies on the regularity issues, in particular the optimal regularity which is achieved when the obstacle is $C^{1,1}$, we opted for the more accurate penalization scheme described in Section \ref{sec.ex}.
\end{rem}

The next result establishes an optimal growth estimate at free boundary points. It is somewhat finer than the previous one. In effect, it states that if the obstacle enjoys $C^{1,\beta}$ regularity and the source term is bounded, then solutions to the $(F, \gamma)$-obstacle problem in $\Omega$ are locally in $C^{1,\alpha}$ (with $\alpha$ given by \eqref{sharpEx}) at free boundary points. By way of interpretation,  the solution leaves the obstacle precisely in a $C^{1, \alpha}$ fashion. As mentioned before, it is of particular interest is the optimal case when $\phi$ is $C^{1,1}$ and $f\equiv 0$, where we obtain quadratic growth for solutions at the free boundary (see Corollary \ref{Cor2}).

\begin{thm}[{\bf Optimal regularity}]\label{MThm1} Let $\beta \in (0, 1]$ and $u$ be a bounded viscosity solution to the $(F, \gamma)$-obstacle problem in $\Omega$ with obstacle $\phi \in C^{1, \beta}(\Omega)$ and $f \in L^{\infty}(\Omega)$ and let $\alpha$ be defined by \eqref{sharpEx}. Then, $u$ is $C^{1, \alpha}$ regular at free boundary points. More precisely, for any $\tilde{\Omega}\subset\subset\Omega$ and $x_0 \in \partial \{u>\phi\} \cap \tilde{\Omega}$ and for $r$ small enough there holds
\begin{equation}\label{EqSRE}
       \displaystyle \sup_{B_r(x_0)} \frac{|u(x)-(u(x_0)+ D u(x_0)\cdot (x-x_0))|}{r^{1+\alpha}}\leq C.\left(\|\phi\|_{C^{1, \beta}(\Omega)}^{\gamma+1}+\|f\|_{L^{\infty}(\Omega)}\right)^{\frac{1}{\gamma+1}},
\end{equation}
where $C>0$ is a universal constant. In particular,
\begin{equation}\label{detach}
\sup_{B_r(x_0)} \frac{|u(x)-\phi(x)|}{r^{1+\alpha}}\leq C^{\ast}.\left(\|\phi\|_{C^{1, \beta}(\Omega)}^{\gamma+1}+\|f\|_{L^{\infty}(\Omega)}\right)^{\frac{1}{\gamma+1}},
\end{equation}
where $C^{\ast}>0$ is a universal constant, i.e. $u$ detaches from the obstacle at the speed dictated by the obstacle's regularity.
\end{thm}

We recall that a constant is called universal if it depends only on the given data (and not on the solution).

As a consequence of the previous Theorem \ref{MThm1} we get, under suitable assumptions on the data, the same regularity for the obstacle problem as for the non-constrained problem:

\begin{cor}\label{Cor1} Let $u$ be a bounded viscosity solutions to $(F, \gamma)$-obstacle problem in $\Omega$ with obstacle $\phi \in C^{1, \frac{1}{\gamma+1}}(\Omega)$ and $f \in L^{\infty}(\Omega)$. Then, $u$ is $C^{1, \frac{1}{\gamma+1}}$ at free boundary points and hence belongs to $C_{\textrm{loc}}^{1, \frac{1}{\gamma+1}}(\Omega)$.
\end{cor}

In particular, as mentioned before, we get the optimal regularity for the homogeneous obstacle problem:

\begin{cor}\label{Cor2} Let $u$ be a bounded viscosity solutions to the homogeneous $(F, \gamma)$-obstacle problem in $\Omega$ (that is $f\equiv 0$) with obstacle $\phi \in C^{1, 1}(\Omega)$. Then, $u$ is $C^{1,1}$ at free boundary points.
\end{cor}

\begin{rem}\label{RemRecessOper}

Regarding the hypothesis of Theorem \ref{MThm1}, we can actually relax the convexity (or concavity) assumption on the nonlinearity $F$. To this purpose, the key ingredient is an available Caffarelli's $W^{2, p}$ regularity theory (see \cite[Theorem 7.1]{C89}) to
$$
  F(D^2u) = f \in L^p \quad \text{in} \quad \Omega \quad (\text{with} \,\,\,p>n).
$$
We recommend the reader to \cite[Section 1]{daS19} and the references therein for a complete state of the art in this direction.

Summarily, recently, in an interesting advancement for this theory, local $C^{1, \alpha}$ (resp. $W^{2, p}$) regularity estimates were improved by Silvestre-Teixeira in \cite[Theorem 1.1]{ST15} and Pimentel-Teixeira \cite[Theorems 1.1 and 6.1]{PimTei16}. In their approach, the novelty with respect to the former results is the concept of \textit{recession} function (the tangent profile for $F$ at ``infinity'') given by
$$
  \displaystyle F^{\ast}(X) \defeq \lim_{\tau \to 0+ } \tau F\left(\frac{1}{\tau}X\right)
$$
In this direction, the authors relaxed the hypothesis of $C^{1,1}$ \textit{a priori} estimates for solutions of the equations without dependence on $x$, by the hypothesis that $F$ is assumed to be ``convex or concave'' only at the ends of $\textit{Sym}(n)$, in other words, when $\|D^2 u\| \approx \infty$. In this context, they proved that if solutions to the homogeneous equation
$$
F^{\ast}(D^2 u) = 0  \quad \textrm{in} \quad B_1
$$
has $C^{1, \alpha_0}$ \textit{a priori} estimates (for some $\alpha_0 \in (0, 1]$), then viscosity solutions to
$$
 	F(D^2 u) = 0  \quad \textrm{in} \quad B_1 \quad (\text{resp.} \,\,\,F(D^2 u) = f)
$$
are of class $C^{1, \alpha}_{\text{loc}}$ for $\alpha <\min\{1, \alpha_0\}$ (resp. are $W_{\text{loc}}^{2, p}$). Among the other works on this topic of investigation, we must quote the sequence of fundamental texts due to Krylov, see \cite{Kry12}, \cite{Kry13} and \cite{Kry17}, which are largely associated with the existence of certain viscosity solutions under either relaxed or no convexity assumptions on $F$ (see also \cite{daSilRic} and references therein for similar results). Such works previously cited constitute further interesting contributions in the modern literature on existence and regularity for viscosity solutions under weaker convexity assumptions in the governing operator.

Finally, if the recession profile associated to $F$, i.e. $F^{\ast}$, enjoys $C^{1,1}$ \textit{a priori} estimates, then a $W^{2, p}$ regularity theory is available to \eqref{Eq1} (see, \cite[Theorem 1.1]{daS19} for more details). For this reason, by following the ideas of \cite[Theorems 5.1, 5.2 and 5.4]{BLOP}, we are able to prove our results to operators under either relaxed or no convexity assumptions on $F$. At the end of the Subsection \ref{Examples}, we will present a number of examples of nonlinearities where our results take place.
\end{rem}

A geometric interpretation of Theorem \ref{MThm1} is the following: if $u$ solves the $(F, \gamma)$-obstacle problem \eqref{Eq1}, and $x_0$ is a free boundary point, then near $x_0$ we obtain
$$
    \displaystyle \sup_{B_r(x_0)} |u(x)|\leq |u(x_0)|+C.r^{1+\alpha}.
$$

On the other hand, from a (geometric) regularity viewpoint, it is a pivotal quantitative information to obtain the counterpart sharp lower estimate. Such a property is denominated \textit{Non-degeneracy} of solutions and we begin our discussion of it in the following Theorem:

\begin{thm}[{\bf Non-degeneracy estimates}]\label{ThmNonDeg} Let $u$ be a bounded viscosity solution to the $(F, \gamma)$-obstacle problem in $\Omega$ with obstacle $\phi \in C^{1,\beta}(\Omega)$ and assume $f\in L^\infty(\Omega)$ is bounded away from zero, i.e. $\displaystyle \inf_{\Omega} f \defeq \mathfrak{m}>0$. Given $\tilde{\Omega}\subset\subset\Omega$ there exists a universal constant $\mathfrak{c}>0$, such that for and $x_0 \in \partial \{u>\phi\} \cap \tilde{\Omega}$ and $r$ small enough
$$
  \displaystyle \sup_{B_r(x_0)} (u(x)-\phi(x_0))\geq \mathfrak{c}.r^{1+\frac{1}{1+\gamma}}
$$
\end{thm}

This result implies a $\frac{1}{\gamma+1}-$growth estimate (on the gradient) away from free boundary points, with an extra correction term given by influence of the gradient of the corresponding obstacle. This is summarized in the following corollary:

\begin{cor}[{\bf Non-degeneracy of the gradient}]\label{CorNonDegGrad} Suppose that the assumptions of Theorem \ref{ThmNonDeg} are in force. If $x_0 \in \{u>\phi\}\cap \tilde{\Omega}$ and $r\defeq \text{dist}(x_0, \partial \{u>\phi\})$, then
$$
\displaystyle \sup_{B_r(x_0)} |D u| \geq \mathfrak{c}.r^{\frac{1}{\gamma+1}} - \frac{1}{2}\|D\phi\|_{L^{\infty}(B_r(x_0))}.
$$
\end{cor}

From a ``free boundary regularity'' perspective, a natural non-degeneracy property states that the solutions of the homogeneous obstacle problem do not decay faster than quadratically close to the free boundary:

\begin{thm}[{\bf Non-degeneracy for the homogeneous problem}]\label{ThmNonDegHom} Let $u$ be a bounded viscosity solution to the $(F, \gamma)$-obstacle problem in $\Omega$ with obstacle $\phi \in C^{2}(\Omega)$ and $f \equiv 0$. Suppose further that
\begin{equation}\label{eq.ob}
|D \phi|^{\gamma}F(D^2 \phi)\leq c<0.
\end{equation}

Given $x_0 \in \{u>\phi\}\cap \tilde{\Omega}$ for $\tilde{\Omega}\subset\subset\Omega$, there exists a universal constant $\mathfrak{c}$ such that for $r$ small enough
$$
  \displaystyle \sup_{B_r(x_0)} (u(x)-\phi(x))\geq \mathfrak{c}.r^{2}
$$
\end{thm}

As consequence of the non-degeneracy of Theorem \ref{ThmNonDegHom} we can show that the free boundary has zero Lebesgue measure in the homogeneous case (and for an obstacle satisfying \eqref{eq.ob}). This requires us to recall the definition of \emph{porosity}: a bounded measurable set $E$ is porous if for any $x\in E$ there exists a $\delta\in(0,1)$ such that for any ball $B_r(x)$ there exists $y\in B_r(x)$ such that
\[
   B_{\delta r}(y)\subset B_r(x)\setminus E.
\]

Notice that if $E$ is porous and $x\in E$ then
\[
   \frac{|B_r(x)\cap E|}{|B_r(x)|}=\frac{|B_r(x)|-|B_r(x)\setminus E|}{|B_r(x)|}\leq 1-\delta^n,
\]
so that $E$ has no points of density one and hence its Lebesgue measure is zero. Now we can state the following corollary:
\begin{cor}\label{Cor3}
Suppose that the assumptions of Theorem \ref{ThmNonDegHom} are in force. Then, the free boundary is porous and in particular it has zero Lebesgue measure.
\end{cor}

In conclusion, as an interesting direction of further research, we remark that the ``singular case'', that is the range $\gamma\in(-1,0)$, could also be considered. We plan to address this case in a forthcoming work, as it does not follow \emph{mutatis mutandis} from the ideas developed in the degenerate case.

The rest of the paper is organized as follows: in Section \ref{sec.ex} we give the appropriate definition of viscosity solutions and proof Theorem \ref{Thm1}, thus providing existence of such solutions for \eqref{Eq1}, together with a basic regularity estimate. In Section \ref{sec.reg} we prove Theorem \ref{MThm1}, together with Corollaries \ref{Cor1} and \ref{Cor2}. At the end, we will present a number of examples where our results take place. In Section \ref{sec.fb} we discuss non-degeneracy, and prove Theorems \ref{ThmNonDeg} and \ref{ThmNonDegHom} and their respective corollaries. Finally, for the reader's convenience, we gather in the Appendix some useful results that were used throughout the paper.

\section{Existence/Uniqueness and basic regularity}\label{sec.ex}

In this section we prove our existence result, namely Theorem \ref{Thm1}. Let us first review the definition of viscosity solution for our operators (see, for instance \cite{CIL}).

\begin{defi}[{\bf Viscosity solutions}] A continuous functions $u \colon \Omega \to\mathbb{R}$ is said to be a viscosity super-solution (resp. viscosity sub-solution) of $G(x, D u, D^2u)= f(x)$ if for every $\varphi \in C^2(\Omega)$ such that $u-\varphi$ has a strict minimum (resp. strict maximum) at the point $x_0$, then
      $$
      G(x_0, \varphi(x_0), D^2 \varphi(x_0)) \leq  f(x_0) \quad (resp. \,\, \geq f(x_0))
      $$
Finally, we say that $u$ is a viscosity solution if it is simultaneously a viscosity sub-solution and a viscosity super-solution.
\end{defi}

It is noteworthy to observe that on the above definition $G$ is asked to be non-decreasing on $X \in Sym(n)$ with respect to the usual partial ordering on symmetric matrices. For this very reason, the classical notions of solution, sub-solution and super-solution are equivalent to the viscosity ones, provided the function is regular enough, \textit{e.g.} $u \in C^2(\Omega)$.

We will make use of the following Comparison Principle, discussed for instance in \cite[Theorem 1.1]{BerDem} and \cite[Theorem 2.1]{BD2}:

\begin{thm}[{\bf Comparison Principle}]\label{comparison principle} Let $\gamma>0$, $F$ satisfying \eqref{eq.pucci}, $f \in C^0(\overline{\Omega})$ and $b$ a continuous increasing function satisfying $b(0)=0$. Suppose $u_1$ and $u_2$ be continuous functions are respectively a viscosity supersolution and subsolution of
\[
|Du|^\gamma F(D^2u)-f(x) = b(u).
\]

If $u_1 \geq u_2$ on $\partial \Omega$, then $u_1 \geq u_2$ in $\Omega$.

If $b$ is nondecreasing (in particular if $b\equiv 0$), the result holds if $u_1$ is a strict supersolution or vice versa if $u_2$ is a strict subsolution. 	
\end{thm}

Now we give a proof of Theorem \ref{Thm1}. While it is achieved by a rather standard penalization argument (cf. \cite{CDV} and references therein), we sketch it here for the sake of completeness:

\begin{proof}[{\bf Proof of Theorem \ref{Thm1}}]

Recall that we want to find (viscosity) solutions of \eqref{Eq1}. This equation can be conveniently be rewritten as the following system:
\begin{equation}\label{eq.obst}
  \left\{ \begin{array}{rcll}
	u & \geq & \phi & \textrm{ in } \Omega \\
    |D u|^\gamma F(D^2u)   & \leq & f & \textrm{ in } \Omega \\
        |D u|^\gamma F(D^2u)  & = & f & \textrm{ in } \Omega\cap\{u>\phi\} \\
  u & = & g & \textrm{ on } \partial \Omega.
\end{array}\right.
\end{equation}

The first step is to consider the following \emph{penalized} problem:
\begin{equation}\label{eq.pen}
  \left\{ \begin{array}{rcll}
   |D u_\varepsilon|^\gamma F(D^2u_\varepsilon)   & = & f+\zeta_\varepsilon(u_\varepsilon-\phi) & \textrm{ in } \Omega \\
  u & = & g & \textrm{ on } \partial \Omega.
\end{array}\right.
\end{equation}
where for each $\varepsilon >0$ we define $\zeta_\varepsilon\in C^\infty(\R)$ to be any smooth strictly increasing function satisfying
\begin{align*}
&\zeta_\varepsilon<0\textrm{ when }t <0, \zeta_\varepsilon(0)=0 \\
&\zeta_\varepsilon(t) = \frac{t}{\varepsilon} \textrm{ when }t\leq -\varepsilon \\
&\zeta_\varepsilon(t) \leq \delta \textrm{ when }t\geq 0
\end{align*}
where $\delta$ is any number in $(0,1)$ (its choice makes no difference to the proof).

Now, as a technical intermediate step, let us consider the truncated version of the above: for $N>0$ we want to define $\zeta_{\varepsilon,N}(t)$ as the maximum between $-N$ and $\zeta_\varepsilon$. To fit into the hypothesis of Theorem \ref{comparison principle} we choose instead $\zeta_{\varepsilon,N}(t)$ to be a strictly increasing that coincides with $\zeta_\varepsilon$ if $\zeta_\varepsilon(t)>-N/2$ and satisfying
\[
   \zeta_{\varepsilon,N}(t)>-N\quad \lim_{t\rightarrow-\infty}\zeta_{\varepsilon,N}(t)=-N.
\]

To get existence of solutions of \eqref{eq.pen} we consider the following problem first:
\begin{equation}\label{eq.pen2}
  \left\{ \begin{array}{rcll}
   |D u_\varepsilon|^\gamma F(D^2u_\varepsilon)   & = & f+\zeta_{\varepsilon,N}(v-\phi) & \textrm{ in } \Omega \\
  u & = & g & \textrm{ on } \partial \Omega.
\end{array}\right.
\end{equation}
where $v$ is some given function (whose smoothness is going to be specified right below). This problem has a unique viscosity solution by Perron's method whenever the right hand side is continuous (see for instance \cite[Theorem 4.1]{CIL}).

Next, notice that thanks to the global results of Birindelli and Demengel in \cite[Theorem 1.1]{BD2}, and the fact that $\zeta_{\varepsilon, N}$ is bounded, there exists $\beta^{\prime}\ll 1$ such that the operator
\[
\mathrm{T}:C^{1,\beta}(\overline{\Omega})\longrightarrow C^{1,\beta}(\overline{\Omega})
\]
given by $\mathrm{T}(v)=u$, where $u$ is the solution of \eqref{eq.pen2} has a fixed point in $C^{1,\beta^{\prime}}(\overline{\Omega})$. Indeed, the ball
\[
  B \defeq \{v\in C^{1,\beta^{\prime}}(\overline{\Omega}):v=g \textrm{ on }\partial \Omega,\|v\|_{C^{1,\beta}(\overline{\Omega})}\leq C\}
\]
is a convex and compact subset of $C^{1,\beta}(\overline{\Omega})$. Moreover, thanks to the \emph{a prior} estimates $\mathrm{T}$ is continuous and, for an appropriate choice of $C$, $\mathrm{T}(B)\subset B$. The existence of a fixed point is therefore guaranteed by Schauder's Fixed Point Theorem (see \cite[Theorem 11.1]{GT}). Let us label $u_\varepsilon$ such a fixed point.

Next, we want to get a uniform bound for $\zeta_{\varepsilon,N}$ so that we are able to get a bound for $u_\varepsilon$ which is uniform in $\varepsilon$. Notice that $\zeta_{\varepsilon,N}$ attains its minimum wherever $u_\varepsilon-\phi$ attains its minimum. Let us thus consider $x_0$ a minimum point for $u_\varepsilon-\phi$ and notice that
\[
  -\infty<(u_\varepsilon-\phi)(x_0)<0 \quad\textrm{ and }\quad x_0\in\Omega.
\]

Since $u_\varepsilon-\phi$ is differentiable, there exists a horizontal plane $\mathrm{P}_{x_0}$ that satisfies
\[
  \mathrm{P}_{x_0}(x)\leq (u_\varepsilon-\phi)(x)\quad\textrm{ in }\quad \Omega\quad\textrm{ and }\quad \mathrm{P}_{x_0}(x_0)= (u_\varepsilon-\phi)(x_0).
\]
In the sequel, we can use $\mathrm{P}_{x_0}+\phi$ as a test function and get
\[
  f(x_0)+\zeta_{\varepsilon,N}((u_\varepsilon-\phi)(x_0))\geq |D(\mathrm{P}_{x_0}+\phi)(x_0)|^\gamma F(D^2(\mathrm{P}_{x_0}+\phi)(x_0))
\]
which translates into
\[
   \zeta_{\varepsilon,N}((u_\varepsilon-\phi)(x_0))\geq |D\phi(x_0)|^\gamma F(D^2\phi(x_0))-f(x_0).
\]
Therefore,
\[
   \zeta_{\varepsilon,N}\geq C
\]
for some constant $C$ depending only on the $\|\phi\|_{C^{1, 1}(\Omega)}$ and $\|f\|_{L^\infty(\Omega)}$. Hence, we can drop the $N$. Further, since $\zeta_\varepsilon$ is increasing and the boundary conditions are independent of $\varepsilon$, we observe that the comparison principle provided in Theorem \ref{comparison principle}, gives that $u_\varepsilon$ is bounded (uniformly on $\varepsilon$).

Now, we take limits. Once again by the results in \cite[Theorem 1.1]{BD2} and the previous observation, the family $\{u_\varepsilon\}_{\varepsilon>0}$ of solutions to \eqref{eq.pen} are uniformly bounded in $C^{1,\beta}(\overline{\Omega})$ and therefore Arzel\`a-Ascoli's theorem ensures the existence of a function $u\in C^{1,\beta^{\prime}}(\overline{\Omega})$ and a subsequence $\{u_{\varepsilon_j}\}_j$ such that
\[
  u_{\varepsilon_j}\longrightarrow u\textrm{ in } C^{1,\beta^{\prime}}(\overline{\Omega})
\]

The bound $|\zeta_\varepsilon(u_\varepsilon-\phi)|\leq C$ ensures that $u\geq \phi$, which is the first equation in \eqref{eq.obst}. The other two equations are immediate by the stability of viscosity solutions under uniform limits (see \textit{e.g.} \cite[Theorem 3.8]{CCKS96}).
\end{proof}

We will finish this section by proving that viscosity solutions to \eqref{Eq1} are unique under a sign constraint on the source term $f$. Notice that due to the way we constructed it (by taking a subsequence) uniqueness is not \emph{a priori} ensured. The next result is a Comparison Principle adjusted to our context. Its proof is a direct consequence of Theorem \ref{comparison principle}.

\begin{lem}[{\bf Comparison Principle}]\label{lemma comparison}
Let $u_1$ and $u_2$ be continuous functions in $\Omega$ and $f \in C^0(\overline{\Omega})$ fulfilling
$$
    |Du_1|^\gamma F(D^2u_1) \leq f(x) \leq |Du_2|^\gamma F(D^2u_2) \quad  \text{ in } \quad \Omega
$$
in the viscosity sense for some uniformly elliptic operator $F$ and $\displaystyle \inf_{\Omega} f>0$ or $\displaystyle \sup_{\Omega} f<0$. If $u_1 \geq u_2$ on $\partial \Omega$, then $u_1 \geq u_2$ in $\Omega$.
\end{lem}

\begin{proof}
Firstly, let us suppose that $\displaystyle \inf_{\Omega} f>0$. For $\varepsilon>0$ fixed, we define $u^{\varepsilon}_2: \bar{\Omega} \to \R$ the auxiliary function given by
$$
 u^{\varepsilon}_2(x) = (1+\varepsilon)^{\frac{1}{1+\gamma}}u_2(x)-\left((1+\varepsilon)^{\frac{1}{1+\gamma}}-1\right)\|u_2\|_{L^{\infty}(\partial \Omega)}.
$$
Hence, we have in the viscosity sense:
$$
\left\{
\begin{array}{rclcl}
   |Du^{\varepsilon}_2|^\gamma F(D^2 u^{\varepsilon}_2) \geq (1+\varepsilon)f(x) & > & f(x) \geq |Du_1|^\gamma F(D^2u_1) &\text{in}& \Omega \\
  u^{\varepsilon}_2 \leq u_2 & \leq  & u_1 & \text{on} & \partial \Omega
\end{array}
\right.
$$
Then, by invoking the Comparison result in Theorem \ref{comparison principle} we obtain that $u^{\varepsilon}_2 \leq u_1$ in $\Omega$ for all $\varepsilon >0$. Finally, letting $\varepsilon \to 0^{+}$, we conclude that $u_2 \leq u_1$.

For the case $\displaystyle \sup_{\Omega} f<0$, we define for $\varepsilon \in (0, 1)$
$$
 u^{\varepsilon}_1(x) = (1-\varepsilon)^{\frac{1}{1+\gamma}}u_1(x)+\left((1-\varepsilon)^{\frac{1}{1+\gamma}}-1\right)\|u_1\|_{L^{\infty}(\partial \Omega)},
$$
and the rest of the proof holds as the previous case.
\end{proof}

\begin{cor}[{\bf Uniqueness}]
The solution found in Theorem \ref{Thm1} is the unique solution to \eqref{Eq1} provided $\displaystyle \inf_{\Omega} f>0$ or $\displaystyle \sup_{\Omega} f<0$.
\end{cor}

\begin{proof}
Let $v$, $w$ be two viscosity solutions to \eqref{Eq1}. Suppose for sake of contradiction that the open set $\mathcal{O}\defeq \{w>v\}$ is nonempty. Since $w>v \geq \phi$ in $\mathcal{O}$, we know that
$$
  |Dw|^{\gamma}F(D^2 w)= f \quad \text{in} \quad \mathcal{O}
$$
in the viscosity sense. For this reason, and since $v$ is by construction a viscosity super-solution, we have
\[
  \left\{ \begin{array}{rcll}
  |Dw|^\gamma F(D^2 w)  & = & f & \textrm{ in } \,\,\,\mathcal{O} \\
  |Dv|^\gamma F(D^2 v)  & \leq & f & \textrm{ in } \,\,\,\mathcal{O} \\
  v & = & w & \textrm{ on } \partial \mathcal{O}.
  \end{array}\right.
\]
By invoking the comparison principle of Lemma \ref{lemma comparison}, we conclude that $v \geq w$ in $\mathcal{O}$, which is a contradiction. Interchanging the roles of $v$ and $w$ in the previous argument we get that the solution of \eqref{Eq1} is unique.
\end{proof}

\begin{rem}
Notice that this uniqueness result is of course consistent with the unconstrained results of \cite{BerDem} and \cite{BD2} (that follow from Theorem \ref{comparison principle}), where the condition of strict super/subsolution is required. In particular, the lack of uniqueness is not a feature coming from the free boundary nature of the problem, as does happen for instance in \cite{CDV}.

In conclusion, we would like to point out though that in the homogeneous setting $f\equiv 0$ we do have uniqueness thanks to the fact that
\[
   |Du|^\gamma F(D^2u)=0\quad \Longrightarrow \quad F(D^2u)=0 \quad (\text{the ``Cutting Lemma''})
\]
as shown in \cite[Appendix]{IS}.
\end{rem}

\section{Theorem \ref{MThm1} and its consequences}\label{sec.reg}

In this section we prove Theorem \ref{MThm1}. We start with the following remarks that will simplify the notation hereafter:

\begin{rem}[{\bf Normalization assumptions}]\label{Rem1} Firstly, we may assume without loss of generality that $u$ solves the $(F, \gamma)$-obstacle problem in $\Omega$ with obstacle $\phi$ and source term $f$ fulfilling
\[
\|\phi\|_{C^{1, \beta}(\Omega)} \leq \frac{1}{2} \quad\textrm{ and }\quad \|f\|_{L^{\infty}(\Omega)}\leq 1.
\]

As a matter of fact, let us consider the normalized function:
\[
v(x) \defeq \frac{u(x)}{\big(2^{1+\gamma}\|\phi\|^{1+\gamma}_{C^{1, \beta}(\Omega)} + \|f\|_{L^{\infty}(\Omega)}\big)^{\frac{1}{1+\gamma}}}.
\]
$v$ thus defined will satisfy an equation like \eqref{Eq1} with $F,f$ and $\phi$ replaced by $\hat{F},\hat{f}$ and $\hat{\phi}$ respectively where
\[
\left\{
\begin{array}{rcl}
  \hat{F}(X) & \defeq & \frac{1}{\big(2^{\gamma+1}\|\phi\|^{\gamma+1}_{C^{1, \beta}(\Omega)} + \|f\|_{L^{\infty}(\Omega)}\big)^{\frac{1}{1+\gamma}}}F\left(\big(2^{\gamma+1}\|\phi\|^{\gamma+1}_{C^{1, \beta}(\Omega)} + \|f\|_{L^{\infty}(\Omega)}\big)^{\frac{1}{1+\gamma}}X\right) \\
  \hat{f}(x) & \defeq & \frac{1}{2^{1+\gamma}\|\phi\|^{1+\gamma}_{C^{1, \beta}(\Omega)} + \|f\|_{L^{\infty}(\Omega)}}f(x) \\
  \hat{\phi}(x) & \defeq & \frac{1}{\big(2^{1+\gamma}\|\phi\|^{1+\gamma}_{C^{1, \beta}(\Omega)} + \|f\|_{L^{\infty}(\Omega)}\big)^{\frac{1}{1+\gamma}}}\phi(x).
\end{array}
\right.
\]
Furthermore, $\hat{F}$ is still elliptic with the same ellipticity constants as $F$, and $\hat{f}$ and $\hat{\phi}$ fall into the desired statements.

This reduction implies, in particular, that we may assume that for any free boundary point $z \in \partial \{u>\phi\}$ we have
\[
   |Du(z)| =|D\phi(z)| \leq \frac{1}{2}.
\]

For the sake of simplicity of notation we will also perform all our estimates in $B_{\frac{1}{2}}$ without loss of generality. Also, throughout the proofs $C$ will denote a universal that may change from line to line.
\end{rem}

In the sequel, the proof will be divided into two technical lemmas, following the strategy set forth in \cite{AdaSRT18}, \cite{ALS15} and \cite{daS19}. The basic idea is to deal with two regimes separately: either a free boundary point belongs to the \textit{singular zone} (see below) where the gradient is small and the operator degenerates, or otherwise the operator is uniformly elliptic and classical results apply.

The singular zone of solutions (see, \cite[Section 3]{AdaSRT18}, \cite{ALS15},  and \cite[Section 1.2]{daS19} for similar ideas) $\mathcal{S}_r^\alpha$ is defined by
$$
   \mathcal{S}_{r}^{\alpha}(B)\defeq \left\{x \in B;\,|D u(x)|\leq r^{\alpha}\right\},
$$
where $0<r \ll 1$ is small, $\alpha$ is defined by \eqref{sharpEx} and $B$ is a ball.

\begin{lem}\label{Lem1} Suppose that the assumptions of Theorem \ref{MThm1} (and Remark \ref{Rem1}) are in force. Let $x_0 \in  \partial \{u>\phi\} \cap \mathcal{S}_{r}^{\alpha}\left(B_{\frac{1}{2}}\right)$ with $0<r <\frac{1}{4}$. Then,
\begin{equation}\label{EqSRE1}
     \displaystyle \sup_{B_r(x_0)} |u(x)-u(x_0)|\leq C.r^{1+\alpha},
\end{equation}
where $C>0$ is a universal constant. In particular
\begin{equation}\label{growth}
 \sup_{B_r(x_0)} |u(x)-(u(x_0)+ D u(x_0)\cdot (x-x_0))|\leq C.r^{1+\alpha}.
\end{equation}
\end{lem}

\begin{proof} First, for a fixed $0<r < \frac{1}{4}$ we define the re-scaled auxiliary functions:
\[
u_{r, x_0}(x) \defeq \frac{u(rx+x_0)-u(x_0)}{r^{1+\alpha}}  \quad \textrm{ and } \quad  \phi_{r, x_0}(x) \defeq \frac{\phi(rx+x_0)-\phi(x_0)}{r^{1+\alpha}}.
\]
Now, notice that $u_{r, x_0}$ is a viscosity solution to the $(F_{r, x_0}, \gamma)-$obstacle, i.e. it satisfies
\begin{equation}\label{EqG-Obs}
|D u_{r, x_0}|^\gamma F_{r, x_0}(D^2 u_{r, x_0})\leq  f_{r, x_0}\quad\textrm{ and }\quad u_{r,x_0}\geq \phi_{r, x_0}
\end{equation}
in $B_1$ where
\[
\left\{
\begin{array}{rcl}
  F_{r, x_0} (X) & = & r^{1-\alpha}F\left(\frac{1}{r^{1-\alpha}} X\right) \\
  f_{r, x_0}(x) & = & r^{1-\alpha(1+\gamma)}f(rx+x_0).
\end{array}
\right.
\]
It is straightforward to check that $F_{r, x_0}$ is still uniformly elliptic (with the same ellipticity constants as $F$). Moreover, it follows by \eqref{sharpEx} that
\[
\|f_{r, x_0}\|_{L^{\infty}(B_1)} \leq 1.
\]

Now, we may estimate the $L^{\infty}-$norm of $\phi_{r, x_0}$ as follows
$$
\begin{array}{rcl}
  \|\phi_{r, x_0}\|_{L^{\infty}(B_1)} & = & \left\|\frac{\phi(rx+x_0)-\phi(x_0)}{r^{1+\alpha}}\right\|_{L^{\infty}(B_1)} \\
   & \leq  & \left\|\frac{\phi(rx+x_0)-\phi(x_0)-rD\phi(x_0) \cdot (x-x_0) }{r^{1+\alpha}}\right\|_{L^{\infty}(B_1)} + \left\|\frac{Du(x_0) \cdot (x-x_0) }{r^{\alpha}}\right\|_{L^{\infty}(B_1)}\\
   & \leq & \frac{3}{2}.
\end{array}
$$
so the function $u_{r, x_0}(x) + \frac{3}{2}$ is non negative and it is still a supersolution.

Hence, by applying the weak Harnack inequality (see equation \eqref{EqWHN} and Theorem \ref{wharn} in the Appendix) and recalling that $u_{r,x_0}(0)=\phi_{r,x_0}(0)$ we obtain

\begin{align}\label{eq.weakhar}
\left\|u_{r, x_0} + \frac{3}{2}\right\|_{L^{p_0}\left(B_{\frac{1}{2}}\right)} & \leq   \displaystyle C.\left(\inf_{B_{1}} \left(u_{r, x_0}(x) + \frac{3}{2}\right)+\|f_{r,x_0}\|_{L^{\infty}(B_1)}\right)	 \nonumber \\
   & \leq  C.\left( \phi_{r, x_0}(0) + \frac{3}{2}+ \|f_{r,x_0}\|_{L^{\infty}(B_1)}\right) \nonumber\\
   &\leq C,
\end{align}
for some universal $p_0>0$ universal.

Now let us consider
\[
 \displaystyle w(x)\defeq \max\left\{u_{r,x_0}(x),\,\sup_{B_1}\phi_{r,x_0}(x)\right\} + \frac{3}{2}.
\]
Notice that $w$ is non-negative sub-solution (in the viscosity sense) to
\[
|Dw|^\gamma F_{r, x_0}(D^2w) = f_{r,x_0}\chi_{\{u_{r,x_0}>\sup\phi_{r,x_0}\}}
\]
and hence by invoking the Local Maximum Principle (see equation \eqref{EqLMP} and Theorem \ref{localmax} in the Appendix) we get that
\[
   \sup_{B_{\frac{1}{4}}}w  \leq \displaystyle C.\left(\left\|w\right\|_{L^{p_0}\left(B_{\frac{1}{2}}\right)}+ \|f_{r,x_0}\|_{L^{\infty}(B_1)}\right).
\]
Combining this estimate with the definition of $w$ and with equation \eqref{eq.weakhar} we get (upon relabeling the constants and some elementary manipulations)
\[
\displaystyle \sup_{B_{\frac{1}{4}}} u_{r, x_0}(x) \leq C.
\]
But on the other hand, since
$$
  u_{r, x_0}(x) \geq - \|\phi_{r, x_0}\|_{L^{\infty}(B_1)} \geq - \frac{3}{2}
$$
we conclude that $u_{r, x_0}$ is uniformly bounded in $B_{\frac{1}{4}}$.

From here we get \eqref{EqSRE1} for $r \in \left(0, \frac{1}{4}\right)$ just by recalling the definition of $u_{r,x_0}$. Since \eqref{growth} follows by the triangle inequality this completes the proof.
\end{proof}

In the following second Lemma, we will analyze the points outside the singular zone, where classical estimates apply.

\begin{lem}\label{Lem2} Suppose that the assumptions of Theorem \ref{MThm1} (and Remark \ref{Rem1}) are in force. Let $z_0 \in  \partial \{u>\phi\} \cap \left(B_{\frac{1}{2}} \setminus \mathcal{S}_{r}^{\alpha}\left(B_{\frac{1}{2}}\right)\right)$ with $0<r < \frac{1}{4}$. Then,
\begin{equation}\label{EqSRE2}
     \displaystyle \sup_{B_{r_{\ast}}(z_0)} |u(x)-u(z_0)|\leq C.r_{\ast}^{1+\alpha},
\end{equation}
for $0<r_{\ast}(r) < \frac{1}{4}$, where $C>0$ is a universal constant.
\end{lem}

\begin{proof} Let $0<r < \frac{1}{4}$ fixed and $z_0 \in  B_{\frac{1}{2}}$ such that $|D u(z_0)|\geq r^{\alpha}$. By taking the particular case $r_{z_0} = \sqrt[\alpha]{|D u(z_0)|}$ we are allowed to apply the Lemma \ref{Lem1} and conclude that
\begin{equation}\label{EqCzse1}
     \displaystyle \sup_{B_{r_{z_0}}(z_0)} |u(x)-u(z_0)|\leq C.r_{z_0}^{1+\alpha}.
\end{equation}
As before, we define the re-scaled auxiliary functions:
$$
\left\{
\begin{array}{rcl}
u_{r_{z_0}, z_0}(x) & \defeq & \frac{u(r_{z_0}x+z_0)-u(z_0)}{r_{z_0}^{1+\alpha}}\\
\phi_{r_{z_0}, z_0}(x) & \defeq & \frac{\phi(r_{z_0}x+z_0)-\phi(z_0)}{r_{z_0}^{1+\alpha}}\\
 f_{r_{z_0}, z_0}(x) & \defeq & r_{z_0}^{1-\alpha(\gamma+1)}f(r_{z_0}x+z_0).
\end{array}
\right.
$$
Notice that
\begin{equation}\label{EqEstGrad}
  |Du_{r_{z_0}, z_0}(0)| = |D\phi_{r_{z_0}, z_0}(0)| = 1 \quad \mbox{and} \quad \|f_{r, x_0}\|_{L^{\infty}(B_1)} \leq 1.
\end{equation}

Moreover, $u_{r_{z_0}, z_0}$ is a viscosity solution to an $(F_{r_{z_0}, z_0}, \gamma)-$obstacle problem as in \eqref{EqG-Obs}. Particularly, $u_{r_{z_0}, z_0}$ fulfills (in the viscosity sense) the following:
$$
  \min\left\{ f_{r_{z_0},z_0}-|D u_{r_{z_0}, z_0}|^{\gamma}F(D^2 u_{r_{z_0}, z_0}), \,\, u_{r_{z_0}, z_0}-\phi_{r_{z_0}, z_0}\right\} = 0.
$$

Now, from the assumption $\|\phi\|_{C^{1, \beta}(B_1)} \leq \frac{1}{2}$ (see Remark \ref{Rem1}) we get
$$
   \|\phi_{r_{z_0}, z_0}\|_{C^{1, \beta}\left(B_{\frac{1}{2}}\right)} \leq C^{\sharp} \,\,\,\mbox{(universal constant)}.
$$
Moreover, \eqref{EqCzse1} assures us that $u_{r_{z_0}, z_0}$ is uniformly bounded in the $L^{\infty}\left(B_{\frac{1}{2}}\right)-$topology. From estimates given in Theorem \ref{Thm1} it follows that
$$
   \|u_{r_{z_0}, z_0}\|_{C^{1, \beta^{\prime}}\left(B_{\frac{1}{2}}\right)} \leq C_1^{\sharp} \quad (\mbox{universal constant}).
$$
Such an estimate and the sentence \eqref{EqEstGrad}, allow us to find a radius $r_0 \ll 1$ (universal) so that
$$
   \mathfrak{c}\leq |D u_{r_{z_0}, z_0}(x)| \leq \mathfrak{c}^{-1} \quad \forall\,\,\,x \in B_{r_0}(z_0)\,\,\,\mbox{and}\,\,\,\mathfrak{c}\in (0, 1)\,\,\,\mbox{fixed}.
$$
Particularly, we obtain the following (in the viscosity sense)
$$
   \min\left\{\mathrm{K}-F(D^2 u_{r_{z_0}, z_0}), \,\, u_{r_{z_0}, z_0}-\phi_{r_{z_0}, z_0} \right\}=0,
$$
where $\mathrm{K} = \mathrm{K}\left(\mathfrak{c}, \alpha, \gamma, \|f_{r_{z_0}, z_0}\|_{L^{\infty}(B_1)}\right)>0$.

Consequently, $u_{r_{z_0}, z_0}$ is a viscosity solution (uniformly bounded) to an obstacle-type problem for a uniformly elliptic operator in $B_{r_0}(z_0)$, with a $C^{1,\beta}$ obstacle ($\phi_{r_{z_0}, z_0}$) and constant (positive) source term $\mathrm{K}$. From theory for obstacle-type problems (see \cite[Theorem 5.4]{BLOP})
$$
   \|u_{r_{z_0}, z_0}\|_{C^{1, \alpha}\left(B_{r_0}\right)} \leq C(\beta, \gamma, \Lambda, \lambda, n)
$$
By scaling back we conclude that
$$
   \|u\|_{C^{1, \alpha}\left(B_{r_0r_{z_0}(z_0)}\right)} \leq C(\beta, \gamma, \Lambda, \lambda, n),
$$
which particularly implies that
$$
\displaystyle \sup_{B_r(z_0)} |u(x)-(u(z_0)+ D u(z_0)\cdot (x-z_0))|\leq C.r^{1+\alpha},
$$
for all $r<r_0r_{z_0}$.
\end{proof}

Finally, we are in position to supply the proof of Theorem \ref{MThm1}.

\begin{proof}[\bf{Proof of Theorem \ref{MThm1}}]  Remember that from Lemma \ref{Lem1} or \ref{Lem2} we have the following:
$$
   \displaystyle \sup_{B_r(y_0)} |u(x)-(u(y_0)+ D u(y_0)\cdot (x-y_0))|\leq C.r^{1+\alpha},
$$
for every $r \in \left(0, \frac{1}{4}\right)$ such that $r^{\alpha}> |Du(y_0)|$ (Lemma \ref{Lem1}) and $r^{\alpha} \leq \sqrt[\alpha]{r_0}|Du(y_0)| = \sqrt[\alpha]{r_0}r_{y_0}$ (Lemma \ref{Lem2}), where $y_0 \in \partial \{u>\phi\} \cap B_{\frac{1}{2}}$. Next prove the desired estimate when
\begin{equation}\label{EqintervR}
  r \in \left(r_0\sqrt[\alpha]{|Du(y_0)|}, \sqrt[\alpha]{|Du(y_0)|}\right).
\end{equation}

For that purpose, suppose that $r$ falls into interval specified in \eqref{EqintervR}. Hence,
$$
\begin{array}{rcl}
  \displaystyle \sup_{B_r(y_0)} |u(x)-(u(y_0)+ D u(y_0)\cdot (x-y_0))| & \leq  & \displaystyle \sup_{B_{r_{y_0}}(y_0)} |u(x)-(u(y_0)+ D u(y_0)\cdot (x-y_0))| \\
   & \leq  & C.r_{y_0}^{1+\alpha}\\
   & \leq & \frac{C}{r_{0}^{1+\alpha}}
   r^{1+\alpha}.
\end{array}
$$
Thus, we obtain the estimate for all $r \in \left(0, \frac{1}{4}\right)$.

This gives the result in $B_{\frac{1}{2}}$. Getting it for $\tilde{\Omega}$ is just a standard covering procedure.  Moreover, taking into account Remark \ref{Rem1}, in order to obtain the desired $C^{1, \alpha}$ regularity estimate for the original solution $u$ (equation \eqref{EqSRE}), one just needs to multiply the constant $C>0$ by the normalization factor $\left(2^{\gamma+1}\|\phi\|_{C^{1, \beta}(\Omega)}^{\gamma+1}+\|f\|_{L^{\infty}(\Omega)}\right)^{\frac{1}{\gamma+1}}$.

Finally, to obtain \eqref{detach} we just compute
$$
\begin{array}{rcl}
  \displaystyle \sup_{B_r(y_0)} |u(x)-\phi(x)| & \leq  & \displaystyle \sup_{B_{r}(y_0)} |u(x)-[u(y_0)+ D u(y_0)\cdot (x-y_0)]|\\
  & + & \displaystyle  \sup_{B_{r}(y_0)} |\phi(x)-[\phi(y_0)+ D \phi(y_0)\cdot (x-y_0)]|\\
   & \leq  & (C+1).r^{1+\alpha}.
\end{array}
$$
\end{proof}

\begin{proof}[{\bf Proof of Corollary \ref{Cor1}}]
We simply note that since $F$ is convex and $\phi\in C^{1,\frac{1}{\gamma+1}}(\Omega)$ then $\alpha=\frac{1}{\gamma+1}$. The result then follows
by combining Theorem \ref{MThm1} with the estimates for the unconstrained problem obtained in \cite[Corollary 3.1]{ART15}.
\end{proof}

\begin{proof}[{\bf Proof of Corollary \ref{Cor2}}]
It follows by noticing that if $f\equiv0$ we can take $\alpha=1$ when making the re-scaling.
\end{proof}

\subsection{Examples}\label{Examples}

We conclude this section presenting some interesting examples where our results hold.

\begin{example} An interesting application of our results when $F(X)=\tr X$ and $f\equiv 1$. In that case, we get the seemingly simple problem
\[
  \left\{ \begin{array}{rcll}
   u & \geq  & \phi & \textrm{ in } \Omega \\
  |Du|^\gamma\Delta u & \leq  & 1 & \textrm{ in } \Omega \\
  |Du|^\gamma\Delta u & = & 1 & \textrm{ in } \Omega\cap\{u>\phi\}. \\
\end{array}\right.
\]
To the best of the authors' knowledge, no results whatever were available for this toy model. According to Theorem \ref{MThm1} and its corollaries, if $\phi \in C^{1, \frac{1}{\gamma+1}}(\Omega)$ our results give $C^{1,\frac{1}{\gamma+1}}$ regularity, which is the optimal regularity of the unconstrained problem, see for instance \cite[Corollary 3.2]{ART15}, \cite[Example 1]{IS} and the references therein. Moreover, the results hold true for a general bounded source term $f$. Further, if $\phi \in C^{1,1}(\Omega)$ then $u\in C^{1,1}$ at free boundary points, hence we recover the optimal result for the classical obstacle problem in this context.
\end{example}

\begin{example} Our results also hold for Pucci's extremal operators (see the Introduction):
$$
F(D^2u)\defeq \mathcal{M}_{\lambda, \Lambda}^{\pm}(D^2 u)
$$
and, more generally, cover Belmann's type equations, which appear in stochastic control problem as an optimal cost:
\[
\displaystyle F(D^2 u) \defeq \inf_{\hat{\alpha} \in \mathcal{A}}\left(L^{\hat{\alpha}} u(x)\right) \quad \left(\text{resp.}\,\,\,\sup_{\hat{\alpha} \in \mathcal{B}}  \left(L^{\hat{\alpha}} u(x)\right)\right),
\]
   and
      $$
         \displaystyle \mathcal{L}^{\hat{\alpha}} u(x) = \sum_{i, j=1}^{n} a_{ij}^{\hat{\alpha}}\partial_{ij} u
      $$
is a family of uniformly elliptic translation invariant operators with ellipticity constants $\lambda$ and $\Lambda$.
\end{example}

\begin{example} As mentioned before in Remark \ref{RemRecessOper} our results hold for operators whose recession profile enjoys of appropriate \textit{a priori} estimates (see, once again, \cite{daS19}, \cite{PimTei16} and \cite{ST15}). By didactic and illustrative reasons, we will exhibit some operators and its recession counterpart. Let $$0 < \sigma_1, \ldots , \sigma_n< \infty.$$
We have the following examples:
\begin{enumerate}
  \item[(E1)]({\bf $m$-momentum type operators}) Let $m$ be an odd number. \textit{The $m$-momentum type operator} given by
  $$
\displaystyle  F_m(D^2 u) = F_m(e_1(D^2 u), \cdots, e_n(D^2 u)) \defeq \sum_{j=1}^{n} \sqrt[m]{\sigma_j^m+e_j(D^2 u)^m}-\sum_{j=1}^{n} \sigma_j.
  $$
defines a uniformly elliptic operator which is neither concave nor convex. Moreover,
$$
  \displaystyle F_m^{\ast}(X) = \lim_{\tau \to 0+} \tau F_m\left(\frac{1}{\tau} X\right) = \sum_{j=1}^{n} e_j(X)
$$
the Laplacian operator.

  \item[(E2)]({\bf Perturbation of ``non-isotropic'' Pucci's operators}) Let us consider
  $$
 F(D^2 u) = F(e_1(D^2 u), \cdots, e_n(D^2 u)) \defeq \sum_{j=1}^{n} \left[h(\sigma_j)e_j(D^2 u) + g(e_j(D^2 u))\right],
  $$
  where $h:[0, \infty) \to \R$ is a continuous function with $h(0) = 0$ and $g: \R \to \R$ is any Lipschitz function such that $g(0) = 0$. Notice that $F$ is uniformly elliptic operator.
  Moreover,
  $$
  \displaystyle F^{\ast}(X) = \lim_{\tau \to 0+} \tau F \left(\frac{1}{\tau} X\right) = \sum_{j=1}^{n} h(\sigma_j)e_j(X),
  $$
which is, up a change of coordinates, the Laplacian operator.

  \item[(E3)]({\bf Perturbation of the Special Lagrangian equation}) Given $h:\R_{+} \to \R$ a continuous function, the \textit{``perturbation'' of the Special Lagrangian equation}
$$
 F(D^2 u) =  F(e_1(D^2 u), \cdots, e_n(D^2 u)) \defeq \sum_{j=1}^{n} \left[h(\sigma_j)e_j(D^2 u) + \arctan(e_j(D^2 u))\right].
$$
defines a uniformly elliptic operator which is neither concave nor convex. Furthermore,
  $$
  \displaystyle F^{\ast}(X) = \lim_{\tau \to 0+} \tau F \left(\frac{1}{\tau} X\right) = \sum_{j=1}^{n} h(\sigma_j)e_j(X),
  $$
which is precisely a ``perturbation'' of the Laplace operator.
\end{enumerate}
\end{example}

\begin{example} Notice that as $F: \textit{Sym}(n) \to \R$, the \textit{recession profile} $F^{\ast}$ should be understood as the limiting equation for the natural scaling on $F$. By way of illustration, for a number of operators, it is possible to check the existence of the limit
$$
  \mathfrak{A}_{ij} \defeq \lim_{\|X\|\to \infty} F_{ij}(X),
$$
where $F_{ij}(X) = \frac{\partial F}{X_{ij}}(X)$. In this situation, $F^{\ast}(X) = \tr(\mathfrak{A}_{ij}X)$. A particular example is the class of Hessian operators of the form:
$$
\displaystyle  F_m(e_1(D^2 u), \cdots, e_n(D^2 u)) \defeq \sum_{j=1}^{n} \sqrt[m]{1+e_j(D^2 u)^m}  - n,
$$
where $m \in \mathbb{N}$ (an odd number). In such a case, $\displaystyle F^{\ast}(X) = \sum_{j=1}^{n} e_j(X)$ (the Laplacian).
\end{example}

\section{Non-degeneracy results and Lebesgue measure of the free boundary}\label{sec.fb}

This section is devoted to prove Theorems \ref{ThmNonDeg} and \ref{ThmNonDegHom} and their corollaries related to geometric non-degeneracy. They play an essential role in the description of solutions to free boundary problems of obstacle-type.

\begin{proof}[{\bf Proof of Theorem \ref{ThmNonDeg}}]
Notice that, due to the continuity of solutions, it is sufficient to prove that such an estimate is satisfied just at point within $\{u>\phi\} \cap \tilde{\Omega}$ for $ \tilde{\Omega}\subset\subset\Omega$.

First of all, for $x_0 \in \{u>\phi\} \cap \tilde{\Omega}$ and $r$ small enough so that $B_r(x_0)\subset\{u>\phi\} \cap \tilde{\Omega}$ let us define the scaled function
$$
    u_{r}(x) \defeq \frac{u(x_0+rx)}{r^{1+\frac{1}{1+\gamma}}} \quad  \mbox{for} \quad  x \in B_1.
$$

Now, let us introduce the comparison function:
$$
   \displaystyle \varphi (x) \defeq \left\{\frac{ \mathfrak{m} \left(\gamma+1 \right)^{\gamma + 2} }{\left[\lambda+n(\gamma+1)\Lambda\right]\left( \gamma+2\right)^{\gamma + 1}}\right\}^{\frac{1}{\gamma+1}} |x|^{1+\frac{1}{1+\gamma}} + \frac{1}{r^{1+\frac{1}{1+\gamma}}}\phi(x_0).
$$

Straightforward calculus shows that
$$
  |D \varphi|^\gamma   \mathcal{G}(D^2 \varphi) - \hat{f}\left( x  \right)\leq 0 \quad \text{in} \quad B_1
$$
and
$$
   |D u_r|^\gamma\mathcal{G}(D^2 u_{r}) - \hat{f}\left( x  \right)= 0  \quad \text{in} \quad B_1\cap \{u_{r}> \phi_{r}\}
$$
in the viscosity sense, where
$$
\left\{
\begin{array}{rcl}
  \mathcal{G}(X) & \defeq & r^{\frac{\gamma}{\gamma+1}}F\left(r^{-\frac{\gamma}{\gamma+1}}X\right) \\
  \hat{f}(x) & \defeq & f(x_0 + rx) \\
  \phi_{r}(x) & \defeq & \frac{\phi(x_0+rx)}{r^{1+\frac{1}{1+\gamma}}}
\end{array}
\right.
$$
Moreover, $\mathcal{G}$ is uniformly elliptic according to \eqref{eq.pucci} and $\displaystyle \inf_{B_1} \hat{f} \geq \inf_{\Omega} f>0$.

Finally, if $u_{r} \leq \varphi$ on $\partial(B_1 \cap \{u_{r}> \phi_{r}\})$ then the Comparison Principle (Lemma \ref{lemma comparison}), would imply that
$$
   u_{r} \leq \varphi \quad \mbox{in} \quad B_1 \cap \{u_{r}> \phi_{r}\},
$$
which clearly contradicts the fact that $u_{r}(0)>\phi_{r}(0)$. Therefore, there exists a point $y \in \partial (B_1 \cap \{u_{r}> \phi_{r}\})$ such that
\[
      u_{r}(y) > \varphi(y).
\]
To conclude, we just notice that (by the choice of $r$) such a point must belong to $\partial B_1\cap \{u_{r}> \phi_{r}\}$ so that then
\[
\frac{1}{r^{1+\frac{1}{1+\gamma}}}.\sup_{B_r(x_0)}u\geq u_r(y)\geq \mathfrak{c}+\frac{1}{r^{1+\frac{1}{1+\gamma}}}\phi(x_0)
\]
and we get the desired result.
\end{proof}

Now, we will prove the non-degeneracy of the gradient.

\begin{proof}[{\bf Proof of Corollary \ref{CorNonDegGrad}}]
In effect, let $x_0 \in \{u>\phi\} \cap B_{\frac{1}{2}}$ and $y_0 \in \partial \{u>\phi\}$ such that
$$
  r \defeq \text{dist}(x_0, \partial \{u>\phi\}) = |x_0 - y_0|.
$$
Now, from non-degeneracy, Theorem \ref{ThmNonDeg}, there exists $z \in \partial B_r(x_0)$ such that
$$
   u(z) - \phi(x_0) = \mathfrak{c}r^{1+\frac{1}{1+\gamma}}.
$$
Moreover, from local Lipschitz regularity of $u$ and $\phi$ we get
$$
\begin{array}{rcl}
  u(z) - \phi(x_0) & = & u(z) - u(y_0) + \phi(y_0) -  \phi(x_0) \\
   & \leq & \|Du\|_{L^{\infty}(B_r(x_0))}|z-y_0| + \|D \phi\|_{L^{\infty}(B_r(x_0))}|y_0-x_0| \\
   & \leq & 2r\|Du\|_{L^{\infty}(B_r(x_0))} + r\|D \phi\|_{L^{\infty}(B_r(x_0))}
\end{array}
$$
Therefore, by using the previous estimates we obtain the desired result
$$
  \displaystyle \|Du\|_{L^{\infty}(B_r(x_0))} \geq \mathfrak{c}.r^{\frac{1}{\gamma+1}} - \frac{1}{2}\|D\phi\|_{L^{\infty}(B_r(x_0))}.
$$
\end{proof}

\begin{rem} An interesting piece of information about previous Corollary \ref{CorNonDegGrad} is the following: when the gradient of the obstacle is ``flat'' enough, i.e. $\|D\phi\|_{L^{\infty}(B_r(x_0))} \ll 1$, then solutions to our obstacle problem present an almost $\frac{1}{\gamma+1}-$growth away from free boundary points (cf. \cite[Theorem 3.3]{ART15} for a similar non-degeneracy property). Furthermore, a legitimate $\frac{1}{\gamma+1}-$behavior is brought to light once we suppose that $\|D\phi\|_{L^{\infty}(B_r(x_0))} \leq \mathfrak{c}_0r^{\frac{1}{\gamma+1}}$ and $\mathfrak{c}>2\mathfrak{c}_0$, which is not a restrictive assumption, due to explicit universal dependence of constant $\mathfrak{c}$.

\end{rem}

Next we will prove our second non-degeneracy result.

\begin{proof}[{\bf Proof of Theorem \ref{ThmNonDegHom}}]

Let $y \in \{u > \phi\}\cap\tilde{\Omega}$ and $v(x) = \phi(x) +\varepsilon|x-y|^2$, where $\varepsilon \ll 1$ is chosen such that $|Dv|^{\gamma}F(D^2 v)<0$, which is possible since the second order degenerate fully nonlinear operator in force is continuous with respect to parameter $\varepsilon$.

Now, by putting $r<\mbox{dist}(x_0, \partial \tilde{\Omega})$, we obtain that
$$
  |Dv|^{\gamma}F(D^2 v)<0 = |Du|^{\gamma}F(x, D^2 u) \quad \mbox{in} \quad \{u>\phi\}\cap B_r(x_0)
$$
in the viscosity sense. Furthermore, $u(y) \geq \phi(y) = v(y)$. By invoking the Comparison Principle (Theorem  \ref{comparison principle}) it follows that there is $z_y\in \partial(\{u>\phi\}\cap B_r(x_0))$ such that $u(z_y) \geq v(z_y)$. Since $u <v$ on $B_r(x_0) \cap \partial \{u >\phi \}$ it must hold that $z_y \in \{u > \phi\} \cap \partial B_r(x_0)$ We conclude the proof by continuity, by letting $y \to x_0$.

\end{proof}

As mentioned before, the porosity of the free boundary is a consequence of the non-degeneracy in the homogeneous case:

\begin{proof}[{\bf Proof of Corollary \ref{Cor3}}]
Let $x_0\in\partial\{u>\phi\}\cap \tilde{\Omega}$ and pick $r$ small enough so that $B_{2r}(x_0)\subset\subset \tilde{\Omega}$. By Theorem \ref{ThmNonDegHom} we have that there exists some $y\in\partial B_r(x_0)$ such that
\begin{equation}\label{porosity1}
u(y)-\phi(y)\geq c.r^2
\end{equation}
for some (universal) constant $c$.

On the other hand, the growth control proved in Theorem \ref{MThm1} gives
\begin{equation}\label{porosity2}
u(y)-\phi(y)\leq C.\big(\textrm{dist}(y,\partial\{u>\phi\})\big)^2.
\end{equation}

\eqref{porosity1} and \eqref{porosity2} together imply
\begin{equation}\label{porosity3}
\textrm{dist}(y,\partial\{u>\phi\})>\left(\frac{c}{C}\right)^{1/2}r=:\tilde{C}.r
\end{equation}
and taking $\delta \defeq \frac{\tilde{C}}{4}$ we obtain the that $B_{2\delta r}(y)\cap B_{2r}(x_0)\subset \{u>\phi\}\cap \tilde{\Omega}$ and the result is proved.
\end{proof}

\section{Appendix}\label{Append}

In this Appendix we gather, for the reader's convenience, the statements of two fundamental results that have been cited in the proof of Lemma \ref{Lem1}, namely the Weak Harnack inequality and the Local Maximum Principle:

\begin{thm}[{\bf Weak Harnack inequality, \cite[Theorem 2]{Imb}}]\label{wharn}
Let $u$ be a non-negative continuous function such that
$$
   F_0(x, Du, D^2u)\leq 0 \quad\textrm{ in }\quad B_1
$$
in the viscosity sense (i.e. $u$ is a viscosity super-solution). Assume that $F_0$ is uniformly elliptic in the $X$ variable (recall \eqref{eq.pucci}) and $F_0 \in C^0(B_1\times \left(\R^N \setminus B_{\mathrm{M}_{\mathrm{F}}}\right)\times \text{Sym}(n))$ for some $\mathrm{M}_{\mathrm{F}} \geq 0$. Further assume that
\begin{equation}\label{eq.elliwh}
   |\xi|\geq \mathrm{M}_{\mathrm{F}} \quad \text{and} \quad F_0(x, \xi, X)\leq 0 \quad \Longrightarrow \quad \mathcal{M}^-_{\lambda,\Lambda}(X)-\sigma(x)|\xi|-f_0(x)\leq 0.
\end{equation}
for continuous functions $f_0$ and $\sigma$ in $B_1$. Then, for any $q > n$
\[
   \|u\|_{L^{p_0}\left(B_{\frac{1}{4}}\right)}\leq C.\left(\inf_{B_{\frac{1}{2}}} u+\max\left\{\mathrm{M}_{\mathrm{F}}, \|f_0\|_{L^n(B_1)}\right\}\right)
\]
for some (universal) $p_0>0$ and a constant $C>0$ depending on $n, q, \lambda, \Lambda$ and $\|\sigma\|_{L^q(B_1)}$.
\end{thm}

\begin{thm}[{\bf Local Maximum Principle, \cite[Theorem 3]{Imb}}] \label{localmax}
Let $u$ be a continuous function satisfying
$$
   F_0(x, Du, D^2u)\geq 0 \quad\textrm{ in }\quad B_1
$$
in the viscosity sense (i.e. $u$ is a viscosity sub-solution). Assume that $F_0$ is uniformly elliptic in the $X$ variable (recall \eqref{eq.pucci}) and $F_0 \in C^0(B_1\times \left(\R^N \setminus B_{\mathrm{M}_{\mathrm{F}}}\right)\times \text{Sym}(n))$ for some $\mathrm{M}_{\mathrm{F}} \geq 0$. Further assume that
\begin{equation}\label{eq.elliplmp}
   |\xi|\geq \mathrm{M}_{\mathrm{F}} \quad \text{and} \quad F_0(x, \xi, X)\geq 0 \quad \Longrightarrow \quad \mathcal{M}^+_{\lambda,\Lambda}(X)+\sigma(x)|\xi|+f_0(x)\geq 0.
\end{equation}
for continuous functions $f_0$ and $\sigma$ in $B_1$. Then, for any $p>0$ and $q > n$
\[
\sup_{B_{\frac{1}{4}}}u\leq C.\left(\|u^+\|_{L^p\left(B_{\frac{1}{2}}\right)}+\max\left\{\mathrm{M}_{\mathrm{F}}, \|f_0\|_{L^n(B_1)}\right\}\right)
\]
where $C>0$ is a universal constant depending on $n, q, \lambda, \Lambda,\|\sigma\|_{L^q(B_1)}$ and $p$.
\end{thm}

Both of these results were proved originally in Imbert's paper \cite{Imb}, following the strategy of the uniformly elliptic case, see \cite[Section 4.2]{CC95}. This strategy is based on the so called $L^\varepsilon$ Lemma, which gives a polynomial decay for the measure of the super-level sets of a nonnegative supersolution for the Pucci extremal operator $\mathcal{M}^+_{\lambda,\Lambda}$ (see the upper cited Caffarelli-Cabr\'{e}'s book for precise statements):
\begin{equation}\label{eq.lep}
\left|\left\{x \in B_1:u(x)>t\right\}\cap B_1\right|\leq C.t^{-\varepsilon}.
\end{equation}

In Imbert's paper, unfortunately, there is a gap in the proof of \eqref{eq.lep}, which is made up for in a subsequent paper with Silvestre, see \cite{IS2}, where an appropriate $L^\varepsilon$ estimate is completely addressed. In fact, their proof holds for ``Pucci extremal operators for large gradients'' defined, for a fixed $\tau$, by:
\[
\widetilde{\mathcal{M}}^+_{\lambda,\Lambda}(D^2u,Du)\defeq
\left\{\begin{array}{ll}
\mathcal{M}^+_{\lambda,\Lambda}(D^2 u)+\Lambda|Du| & \text{ if }|Du|\geq \tau \\
+\infty & \text{ otherwise }
\end{array}\right.
\]
\[
\widetilde{\mathcal{M}}^-_{\lambda,\Lambda}(D^2u,Du)\defeq
\left\{\begin{array}{ll}
\mathcal{M}^-_{\lambda,\Lambda}(D^2 u)-\Lambda|Du| & \text{ if }|Du|\geq \tau \\
-\infty & \text{ otherwise }
\end{array}\right.
\]
The $L^\varepsilon$ estimate is proved to hold whenever $\tau\leq \varepsilon_0$ universal (see, \cite[Theorem 5.1]{IS2}). Notice that the ellipticity condition $\widetilde{\mathcal{M}}^-_{\lambda,\Lambda}$ is consistent with \eqref{eq.elliwh} if we take $\sigma(x)\equiv\Lambda$; more precisely, if \eqref{eq.elliwh} and $u$ is a supersolution for $F_0$, then it is also a supersolution for $\widetilde{\mathcal{M}}^-_{\lambda,\Lambda}$ with right hand side $f_0$. An analogous reasoning is valid for $\widetilde{\mathcal{M}}^+_{\lambda,\Lambda}$ and \eqref{eq.elliplmp}.

Once the $L^\varepsilon$ is obtained, the proof of Theorem \ref{wharn} is exactly as the one in \cite{Imb} which is, in turn, a modification of the uniformly elliptic case in \cite[Theorem 4.8, a]{CC95}. As for Theorem \ref{localmax}, it also follows from \eqref{eq.lep} by first assuming that the $L^\varepsilon$ norm of $u^+$ is small and the obtaining the general result by interpolation. Indeed, the smallness of the $L^\varepsilon$ norm readily implies \eqref{eq.lep} which in turn gives $u$ is bounded (see, \cite[Lemma 4.4]{CC95}, which is adapted in \cite[Section 7.2]{Imb}).

Our class of operators fits in this scenario by setting
\[
   F_0(x,Dv,D^2v)=|Dv|^\gamma F(D^2v)-f(x).
\]
and with
\[
f_0(x)\defeq \frac{\|f\|_{L^\infty(B_1)}}{\varepsilon_0^\gamma} \quad \text{for suitable} \quad \varepsilon_0>0.
\]
Indeed, we have that whenever
\[
|Dv|^\gamma F(D^2v)\leq f(x) \quad \text{in} \quad B_1
\]
in the viscosity sense, the ellipticity of $F$ ensures that
\[
\mathcal{M}^-_{\lambda,\Lambda}(D^2v)\leq F(D^2v)\leq \frac{f(x)}{|Dv|^\gamma}
\]
whenever $|Du|\geq\mathrm{M}_{\mathrm{F}} = \varepsilon_0>0$ so that
\[
\mathcal{M}^-_{\lambda,\Lambda}(D^2v)-\Lambda|Dv|-f_0(x)\leq \left(\frac{1}{|Dv|^\gamma}-\frac{1}{\varepsilon_0^\gamma}\right)\|f\|_{L^\infty(B_1)}\leq 0.
\]

Now setting
\[
v(x) \defeq u(x)+\varepsilon_0\quad \text{with } \,\,\, \varepsilon_0\in \left(0, \frac{\mathrm{T}_0}{2}\right]
\]
for some suitable $\mathrm{T}_0>0$ to be chosen \textit{a posteriori}. Moreover, the constants obtained in \cite{IS2} are monotone with respect to $\tau$ and bounded away from 0 and $+\infty$, so we get a uniform estimate as \eqref{eq.lep} for supersolutions of $\mathcal{G}[v]\defeq |Dv|^\gamma F(D^2v)$.

Therefore, in such a situation we have (recall $\sigma(x)\equiv\Lambda$) from Theorem \ref{wharn}:
 {\scriptsize{
 \begin{equation}\label{EqWHN}
      \|u\|_{L^{p_0}\left(B_{\frac{1}{4}}\right)}\leq\left\{
 \begin{array}{rcl}
   \displaystyle C.\left\{\inf_{B_{\frac{1}{2}}} u + \left[(\gamma+1) \sqrt[n]{|B_1|}\|f\|_{L^{\infty}(B_1)}\right]^{\frac{1}{\gamma+1}}\right\} & \text{if} & \varepsilon_0 \leq \frac{ \sqrt[n]{|B_1|}\|f\|_{L^{\infty}(B_1)}}{\varepsilon_0^{\gamma}} \\
   \displaystyle C.\left(\inf_{B_{1}} u + \mathrm{T}_0\right) & \text{if} & \varepsilon_0 > \frac{ \sqrt[n]{|B_1|}\|f\|_{L^{\infty}(B_1)}}{\varepsilon_0^{\gamma}},
 \end{array}
 \right.
 \end{equation}}}
where we have used in the first inequality that the function $\mathfrak{h}(t) = t + \frac{ \sqrt[n]{|B_1|}\|f\|_{L^{\infty}(B_1)}}{t^{\gamma}}$ for $t>0$ is optimized (its lowest upper bound) when $t_0 = \left( \gamma\sqrt[n]{|B_1|}\|f\|_{L^{\infty}(B_1)}\right)^{\frac{1}{\gamma+1}}$.

In conclusion, if $1\leq p_0\leq \infty$ we have from Minkowski's Inequality
$$
  \|u\|_{L^{p_0}\left(B_{\frac{1}{2}}\right)} \leq \|v\|_{L^{p_0}\left(B_{\frac{1}{2}}\right)} + C(p_0, n, \mathrm{T}_0)\leq C(\text{RHS in} \,\,\eqref{EqWHN}, p_0, n, \mathrm{T}_0).
$$
Otherwise, if $0< p_0<1$ by using the inequality (see, \cite{Day40})
$$
   \|\omega_1+\omega_2\|_{L^{p_0}(\Omega)}\leq 2^{\frac{1-p_0}{p_0}}\left(\|\omega_1\|_{L^{p_0}(\Omega)}+\|\omega_2\|_{L^{p_0}(\Omega)}\right)
$$
we get
$$
   \|u\|_{L^{p_0}\left(B_{\frac{1}{2}}\right)} \leq C(\text{RHS in} \,\,\eqref{EqWHN}) + C(p_0, n, \mathrm{T}_0),
$$
thereby obtaining the desired estimate.

Similarly, for Theorem \ref{localmax}, if
\[
|Du|^\gamma F(D^2u)\geq f(x) \quad \text{in} \quad B_1
\]
in the viscosity sense we again have
\[
\mathcal{M}^+_{\lambda,\Lambda}(D^2u)\geq F(D^2u) \geq \frac{f(x)}{|Du|^{\gamma}} \geq -\frac{\|f\|_{L^{\infty}(B_1)}}{|Du|^{\gamma}}\,\,\,\text{whenever}\,\,\, \varepsilon_0  = \mathrm{M}_{\mathrm{F}} \leq |Du|,
\]
and we can set, similarly as above, $f_0(x)\defeq \frac{\|f\|_{L^{\infty}(B_1)}}{\varepsilon_0^{\gamma}}
$
to get
\[
\mathcal{M}^+_{\lambda,\Lambda}(D^2v)+\Lambda|Dv|+f_0(x)\geq \left(\frac{1}{\varepsilon_0^\gamma}-\frac{1}{|Du|^\gamma}\right)\|f\|_{L^\infty(B_1)} \geq 0,
\]
where $v(x) \defeq u(x)+\frac{\varepsilon_0}{\max\left\{1, 2^{\frac{1-p}{p}}\right\}|B_1|^{\frac{1}{p}}}\quad \text{with } \,\,\, \varepsilon_0\in \left(0, \frac{\mathrm{T}_0}{2}\right]$.

Therefore, in such a setting we have from Theorem \ref{localmax}:
$$
\begin{array}{ccl}
 \displaystyle \sup_{B_{\frac{1}{2}}}u & \leq & \displaystyle \sup_{B_{\frac{1}{2}}} v\\
  & \leq & C.\left(\max\left\{1, 2^{\frac{1-p}{p}}\right\}\|u^+\|_{L^p(B_1)}+\varepsilon_0+ \max\left\{\varepsilon_0, \sqrt[n]{|B_1|}\frac{\|f\|_{L^{\infty}(B_1)}}{\varepsilon_0^{\gamma}}\right\}\right)\\
   & \defeq & \Xi,
\end{array}
$$
where, as before, we can estimate
{\scriptsize{
\begin{equation}\label{EqLMP}
\Xi\leq\left\{
 \begin{array}{rcl}
   \displaystyle C.\left\{\max\left\{1, 2^{\frac{1-p}{p}}\right\}\|u^+\|_{L^p(B_1)} + \left[(\gamma+1) \sqrt[n]{|B_1|}\|f\|_{L^{\infty}(B_1)}\right]^{\frac{1}{\gamma+1}}\right\} & \text{if} & \varepsilon_0 \leq \frac{ \sqrt[n]{|B_1|}\|f\|_{L^{\infty}(B_1)}}{\varepsilon_0^{\gamma}} \\
   \displaystyle C.\left(\max\left\{1, 2^{\frac{1-p}{p}}\right\}\|u^+\|_{L^p(B_1)} + \mathrm{T}_0\right) & \text{if} & \varepsilon_0 \geq \frac{ \sqrt[n]{|B_1|}\|f\|_{L^{\infty}(B_1)}}{\varepsilon_0^{\gamma}},
 \end{array}
 \right.
\end{equation}}}
thereby finishing the analysis.

\begin{rem}
Let us stress that for a source term with sign constraint, i.e. $f(x) \geq 0$, we are able to obtain a more direct estimate. Indeed, if
\[
|Du|^\gamma F(D^2u)\geq f(x) \quad \text{in} \quad B_1
\]
in the viscosity sense, we have
\[
\mathcal{M}^+_{\lambda,\Lambda}(D^2u)\geq F(D^2u) \geq \frac{f(x)}{|Du|^{\gamma}} \geq \left(\frac{\mathrm{T}_0-\varepsilon_0}{\mathrm{T}_0^2}\right)^\gamma f(x),
\]
whenever $\varepsilon_0 = \mathrm{M}_{\mathrm{F}} \leq |Du| \leq \frac{\mathrm{T}_0^2}{\mathrm{T}_0-\varepsilon_0}$ and we can set, $f_0(x)\defeq -\left(\frac{T_0-\varepsilon_0}{T_0^2}\right)^\gamma f(x)$
to get
\[
   \mathcal{M}^+_{\lambda,\Lambda}(D^2u)+\Lambda|Du|+f_0(x)\geq 0.
\]

Therefore, we have from Theorem \ref{localmax}:
$$
\begin{array}{ccl}
 \displaystyle \sup_{B_{\frac{1}{2}}}u & \leq & C.\left(\|u^+\|_{L^p(B_1)}+\max\left\{\varepsilon_0, \left(\frac{T_0-\varepsilon_0}{T_0^2}\right)^\gamma\|f\|_{L^n(B_1)}\right\}\right)\\
   & \leq & C.\left(\|u^+\|_{L^p(B_1)}+\max\left\{\frac{\mathrm{T}_0}{2}, \frac{1}{\mathrm{T}_0^{\gamma}}\|f\|_{L^{n}(B_1)}\right\}\right)
\end{array}
$$
\end{rem}

In conclusion, taking into account that we wish to prove regularity estimates along free boundary points $x_0 \in \partial \{u>\phi\}\cap \Omega$ and $|Du(x_0)|=|D\phi(x_0)|$, it is enough to choose $\mathrm{T}_0 = \|\phi\|_{C^{1, \beta}(B_1)}+1$. Furthermore, in both previous estimates, namely \eqref{EqWHN} and \eqref{EqLMP} the corresponding involved constants do not degenerate as $\varepsilon_0$ goes to zero.

\subsection*{Acknowledgments.} This work was partially supported by Consejo Nacional de Investigaciones Cient\'{i}ficas y T\'{e}cnicas (CONICET-Argentine) and Coordena\c{c}\~{a}o de Aperfei\c{c}oamento de Pessoal de N\'{i}vel Superior (PNPD/UnB-Brazil) Grant No. 88887.357992/2019-00 and CNPq-Brazil under Grant No. 310303/2019-2. J.V. da Silva thanks FCEyN/CEMIM from Universidad Nacional de Mar del Plata for its warm hospitality and for fostering a pleasant scientific atmosphere during his visit where part of this work was written.

The authors would like to thank the anonymous Referee whose insightful comments and suggestions helped make this paper more clear and mathematically accurate.

\end{document}